\documentclass[12pt,reqno]{amsart}
\usepackage[a4paper,margin=2.4cm]{geometry}
\usepackage{amssymb,amsmath,amsthm,mathtools}
\usepackage{enumerate}
\usepackage{cases}

\usepackage[colorlinks=true, linkcolor=blue, urlcolor=magenta, hypertexnames=false]{hyperref}

\renewcommand{\Re}{\operatorname{Re}}

\newcommand{\N}{{\mathbb N}}
\newcommand{\Z}{{\mathbb Z}}
\newcommand{\R}{{\mathbb R}}
\newcommand{\C}{{\mathbb C}}

\newcommand{\dd}{\mathrm d}
\newcommand{\ii}{\mathrm i}
\newcommand{\ee}{\mathrm e}

\newcommand{\g}{{\mathfrak{g}}}

\numberwithin{equation}{section}

\theoremstyle{plain}
\newtheorem{theorem}{\bf Theorem}[section]
\newtheorem*{theorem*}{Theorem}
\newtheorem{lemma}[theorem]{\bf Lemma}

\newtheorem*{proposition*}{\bf Proposition}

\theoremstyle{definition}
\newtheorem{definition}[theorem]{\bf Definition}
\newtheorem*{definition*}{\bf Definition}

\theoremstyle{remark}
\newtheorem*{remark*}{\bf Remark}
\newtheorem{remark}[theorem]{\bf Remark}

\newcommand{\ddiv}{\mathrm{div}}

\newcommand{\grad}{{\nabla}}
\newcommand{\lap}{{\Delta}}

\newcommand{\abs}[1]{\left|#1\right|}

\newcommand{\paren}[1]{\left( #1 \right)}

\newcommand{\pochhammer}[2]{\left( #1 \right)_{#2}}
\newcommand{\floor}[1]{\lfloor #1 \rfloor}
\newcommand{\ceil}[1]{\lceil #1 \rceil}
  
\begin{document}

\title[Optimal fractional discrete Hardy inequalities]{Optimal fractional discrete Hardy inequalities \\ on the half-line}

\author{Franti\v{s}ek \v{S}tampach}
\address{Department of Mathematics, Faculty of Nuclear Sciences and Physical Engineering, Czech Technical University in Prague, Trojanova 13, 12000 Prague~2, Czech Republic.}
\email{stampfra@cvut.cz}

\author{Jakub Waclawek}
\address{Department of Mathematics, Faculty of Nuclear Sciences and Physical Engineering, Czech Technical University in Prague, Trojanova 13, 12000 Prague~2, Czech Republic.}
\email{waclajak@cvut.cz}



\date{\today}

\begin{abstract}
We consider a Toeplitz realisation of the fractional discrete Laplacian $(-\lap)^{\alpha}$ on the half-line $\N$ as a compression of the full-line fractional discrete Laplacian to $\ell^{2}(\N)$. For all $\alpha>0$, we prove that the fractional Hardy inequality
\[
 (-\lap)^{\alpha}\geq\frac{4^{\alpha}\Gamma^2(\alpha+1/2)}{\pi}\frac{\Gamma(2\,\cdot\,-1)}{\Gamma(2\,\cdot\,-1+2\alpha)}
\]
holds on $\ell^{2}(\N)$ and is optimal in a strong sense. In particular, we show that the inequality cannot be improved and equality is not attained by any nonzero element of $\ell^{2}(\N)$. As a consequence, we deduce a fractional generalisation of the discrete Birman inequality.
\end{abstract}

\maketitle

\section{Introduction and main results} 

The history of the classical \textit{Hardy inequality} dates back to the beginning of the twentieth century~\cite{kuf-mal-per_06}. In pursuit of an elementary proof of the \textit{Hilbert inequality}, Hardy discovered his celebrated inequality, which may be viewed as a lower bound for the discrete Laplace operator on the half-line. More precisely, Hardy's original inequality states
\[
 -\lap_{\N}\geq\rho^{\mathrm{H}}, \quad \mbox{ with } \rho_n^{\mathrm{H}}\coloneq\frac{1}{4n^{2}},
\]
understood in the sense of quadratic forms on $\ell^{2}(\N)$, where the half-line discrete Laplacian is defined by the formula $\lap_\N u_n \coloneq u_{n-1} - 2 u_n + u_{n+1}$ for all $n\in\N$, with the convention $u_0 \coloneq 0$.

Analogous quadratic form inequalities $-\lap\ge \rho$, where $\rho$ denotes the operator of multiplication by a non-negative \textit{Hardy weight}, remain true in various Hilbert spaces and have found a wide range of applications in areas such as probability theory, the theory of partial differential equations, or mathematical physics. In the latter domain, the Hardy inequality, for example, reflects the transience of Brownian motion described via the heat equation. It is a classical result that a Brownian particle in $\R^d$ almost surely visits every bounded set only finitely many times if $d\ge3$. Conversely, Brownian motion in $\R$ and $\R^2$ is recurrent. In parallel, the self-adjoint realisation of the Laplacian on the Lebesgue space $L^2(\R^d)$ is \textit{subcritical} (that is, there exists a non-trivial Hardy weight $\rho\geq0$ such that $-\lap_{\R^d} \ge \rho$) whenever $d\ge3$, whereas it is \textit{critical} (i.e. not subcritical) for $d\in\{1,2\}$.

Similar criticality transitions arise also in the context of the \textit{fractional Hardy inequalities} $(-\lap)^\alpha\ge \rho$ for a positive (non-integer) power $\alpha$ of the Laplace operator. These operators play a fundamental role in the analysis of anomalous transport and diffusion processes~\cite{Meerschaert-diffusion}, and they also appear in a range of physical models, including, for instance, the relativistic Schrödinger-type equation~\cite{Mazya-Schrodinger}.

In \cite{Herbst-frac_lap}, it was observed that $(-\lap_\R)^\alpha$ is subcritical in $L^2(\R)$ if and only if $\alpha\in(0,1/2)$. For this range of the parameter $\alpha$, the inequality
$(-\lap_\R)^\alpha\ge \rho^{\text{He}}_{\alpha}$ holds with the weight
\[
    \rho^{\mathrm{He}}_{\alpha}(x) 
      \coloneq 4^\alpha \frac{\Gamma^{2}\!\paren{\frac{1+2\alpha}{4}}}{\Gamma^{2}\!\paren{\frac{1-2\alpha}{4}}} \frac{1}{\abs{x}^{2\alpha}},
\]
where $\Gamma$ is the Euler gamma function. The optimality of the constant was proved in~\cite{Yafaev_99-sharp_frac_lap}. 

The discrete analogy on $\Z$ resembles the continuous setting. The discrete fractional Laplacian $(-\lap_{\Z})^{\alpha}$ is subcritical in $\ell^{2}(\Z)$ if and only if $\alpha\in(0,1/2)$. For these powers $\alpha$, the fractional Hardy inequality $(-\lap_{\Z})^{\alpha}\geq\rho^{\mathrm{CR}}_{\alpha}$ has been deduced in \cite{cia-ron_18} with
\[
 \rho_{\alpha}^{\mathrm{CR}}(n)\coloneq 4^\alpha \frac{\Gamma^2\!\left(\frac{1+2\alpha}{4}\right)}{\Gamma^2\!\left(\frac{1-2\alpha}{4}\right)}\frac{\Gamma \left(|n| + \frac{1-2 \alpha}{4}\right) \Gamma \left(|n| + \frac{3-2 \alpha}{4}\right)}{\Gamma \left(|n| + \frac{1+2 \alpha}{4}\right) \Gamma \left(|n| + \frac{3+2 \alpha}{4}\right)}.
\]
Moreover, the weight $\rho^{\mathrm{CR}}_{\alpha}$ turns out to be \emph{optimal} (see Definition \ref{def:optimality}) for all $\alpha\in(0,1/2)$, as proved later in \cite{kel-nit_23}. 
The sequence $\rho_{\alpha}^{\mathrm{CR}}$ is asymptotically equivalent to $\rho^{\text{He}}_\alpha(n)$ for $\abs{n}\to\infty$. 

The situation on the half-line is more subtle, as there exist at least two non-equivalent definitions of general positive powers of the Laplace operator. In particular, due to the inherently nonlocal nature of fractional derivatives, the definition depends on whether the functions, resp. sequences, are formally extended by zero to the negative half-axis or are instead treated without such an extension.

In the continuum, both options were studied in~\cite{bogdan-dyda_11-frac_lap} for $\alpha\in(0,1)$. Regarding the latter setting, the fractional powers of the half-line Laplacian $(-\lap_{\R_+})^\alpha$ satisfy the inequality
\begin{equation}
    \label{eq:cont_frac_hardy_on_R_++}
    \int_0^\infty \overline{\varphi(x)} (-\lap_{\R_+})^\alpha\varphi(x)\dd x 
      \ge \frac{\Gamma^2(\alpha+1/2) - \Gamma(2\alpha)\sin(\pi\alpha)}{\pi}\int_0^\infty \frac{\abs{\varphi(x)}^2}{x^{2\alpha}} \dd x
\end{equation}
for all smooth functions $\varphi\in C_0^\infty(\R_+)$ with compact support in $\R_+\equiv(0,\infty)$. Notice that while the constant is sharp for all $\alpha\in(0,1)$, it is only positive if $\alpha\neq1/2$. The case $\alpha=1/2$ was studied, for example, in~\cite{roy-sahu_26-crit_frac_lap}. On the other hand, if we consider the test functions $\varphi\in C_0^\infty(\R_+)$ to be formally extended by zero to the negative half-axis, the fractional Hardy inequality
\begin{equation}
    \label{eq:cont_frac_hardy_on_R_+}
    \int_0^\infty \overline{\varphi(x)} (-\lap_\R)^\alpha\varphi(x)\dd x 
      \ge \frac{\Gamma^2(\alpha+1/2)}{\pi}\int_0^\infty \frac{\abs{\varphi(x)}^2}{x^{2\alpha}} \dd x,
\end{equation}
holds with a different optimal constant, which is positive for all $\alpha\in(0,1)$.

\subsection{The state of the art}
In this paper, we focus on analogous half-line fractional Hardy inequalities in the discrete setting, where the literature is less complete. The half-line discrete fractional Laplacian $(-\lap_\N)^\alpha$ was initially introduced and investigated in~\cite{gerhat-krejcirik-stampach_23_criticality}, where the authors proved that $(-\lap_\N)^\alpha$ is subcritical if and only if $\alpha\in(0,3/2)$ and, within this range, they identified explicit non-optimal Hardy weights. Optimal Hardy weights for $(-\lap_\N)^\alpha$ were found only recently for $\alpha \in (0,1]$ in the form
\begin{equation}
    \label{eq:def:DFF_weight}
    \rho^{\rm DF}_\alpha(n)
      \coloneq4^\alpha \frac{\Gamma^{2}\!\paren{\frac{3+2\alpha}{4}}}{\Gamma^{2}\!\paren{\frac{3-2\alpha}{4}}} \frac{\Gamma\!\paren{n+\frac{5-2\alpha}{4}} \Gamma\!\paren{n-\frac{1+2\alpha}{4}}}{\Gamma\!\paren{n+\frac{5+2\alpha}{4}} \Gamma\!\paren{n-\frac{1-2\alpha}{4}}}
\end{equation}
in~\cite{das-fernandez_26-frac_hardy}. Interestingly, these weights exhibit the asymptotic decay $n^{-2\alpha}$, as $n\to\infty$, for every $\alpha\in(0,1]$, including the critical value $\alpha=1/2$, and the optimal constant differs from its continuous counterpart \eqref{eq:cont_frac_hardy_on_R_++}.

Concerning the discrete Laplacian on the half-line defined as a compression of the full-line operator $-\lap_{\Z}$, see~\eqref{eq:def-frac-lap_abstract} for the definition, the existing literature has addressed only the integer power inequalities, i.e. discrete Hardy--Rellich--Birman inequalities. Hereafter, we denote this half-line discrete Laplacian simply by $-\lap$, omitting any index for convenience. The derivation of the classical Hardy--Rellich--Birman weights $\rho$, satisfying $(-\lap)^\ell\ge\rho$ on $\ell^2(\N)$, for $\ell\in\N$, which mirror the continuous weights in \eqref{eq:cont_frac_hardy_on_R_+}, was first presented in \cite{huang-ye_24-hardy}. The determination of optimal weights was subsequently carried out in \cite{stampach-waclawek_24-birman} and extended to the $\ell^p$-setting in~\cite{stamp_wacl_26-dics_p-birman} and particularly in \cite{stamp_wacl_26-opt_p-birman}, where the authors discovered optimal higher-order Hardy inequalities of the simple form
\begin{equation}
    \label{eq:stamp_wacl_opt_disc_birm}
    \sum_{n=1}^\infty \overline{u_n} (-\lap)^\ell u_n
      \ge\frac{\Gamma^2(\ell+1/2)}{\pi}\sum_{n=1}^\infty\frac{\abs{u_n}^2}{(n-1/2)n(n+1/2)\cdots(n+\ell-1)}
\end{equation}
for all $\ell\in\N$ and $u\in\ell^{2}(\N)$.

The purpose of this paper is to generalise the methods employed in \cite{stampach-waclawek_24-birman,stamp_wacl_26-opt_p-birman} to the framework of fractional Hardy inequalities and to deduce optimal inequalities of the form \eqref{eq:stamp_wacl_opt_disc_birm} for the operator $(-\lap)^\alpha$ with arbitrary $\alpha>0$.

\subsection{Organisation}

In Section \ref{subsec:main_results}, we state our main results (Theorems \ref{thm:opt_frac_hardy} and \ref{thm:clas_frac_hardy}). Section \ref{sec:prelim} introduces the necessary notation and defines the operator $(-\lap)^\alpha$ for $\alpha>0$. In Section \ref{sec:abstract_results}, we formulate our complementary results and develop an abstract method to construct fractional Hardy weights from a given \textit{parameter sequence} $\g$ that satisfies certain sufficient conditions. Lastly, Section \ref{sec:proof_main_results} proves the main results by applying the abstract theory to a specific choice of the parameter sequence.

\subsection{Main results}
\label{subsec:main_results}

Before stating our main results, we recall two definitions of fractional Hardy weights and their optimality adopted from \cite{stamp_wacl_26-opt_p-birman,hake-keller-pogorzelski_26_frac_lap_graph}. The exact definition of the operator $(-\lap)^{\alpha}$ is postponed to \eqref{eq:def-frac-lap_abstract}.

\begin{definition}
    Let $\alpha>0$. A non-negative sequence $\rho=\{\rho_n\}_{n=1}^\infty$ is called a \textit{discrete fractional Hardy weight} if and only if the \textit{discrete fractional Hardy inequality}
    \begin{equation}
        \label{eq:def:frac_hardy_weight}
        \sum_{n=1}^\infty\overline{u_n}(-\lap)^\alpha u_n 
          \ge \sum_{n=1}^\infty \rho_n\abs{u_n}^2
    \end{equation}
    holds for all compactly supported sequences $u\in C_0(\N)$.
\end{definition}

\begin{remark}
    \label{rem:def:frac_hardy_weight}
Since $(-\lap)^\alpha$ determines a bounded operator on $\ell^2(\N)$ with $\|(-\lap)^\alpha\|\leq4^\alpha$, the multiplication operator by a discrete fractional Hardy weight $\rho$ must also be bounded, and the inequality \eqref{eq:def:frac_hardy_weight} straightforwardly extends from $C_0(\N)$ to $\ell^2(\N)$.
\end{remark}

\begin{definition}
    \label{def:optimality}
    A discrete fractional Hardy weight $\rho$ is said to be \textit{optimal} if it exhibits the following three properties:
    \begin{enumerate}[(i)]
        \item 
            \textit{Criticality}: The weight $\rho$ is called \textit{critical} if for any fractional Hardy weight $\tilde{\rho}$, such that $\tilde{\rho}_n\ge\rho_n$ for all $n\in\N$, it follows that $\tilde{\rho}=\rho$.
        \item
            \textit{Non-attainability}: The weight $\rho$ is called \textit{non-attainable} if, whenever the equality in~\eqref{eq:def:frac_hardy_weight} is attained for $u\in\ell^2(\N)$, necessarily $u\equiv0$.
        \item 
            \textit{Optimality near infinity}: The weight $\rho$ is called \textit{optimal near infinity} if for any $M\ge 1$ and $\varepsilon>0$ there exists $u\in C_0(\N)$ such that $u_n=0$ for all $n<M$ and
        \begin{equation*}
            \sum_{n=1}^\infty\overline{u_n}(-\lap)^\alpha u_n 
              < (1+\varepsilon) \sum_{n=1}^\infty \rho_n\abs{u_n}^2.
        \end{equation*}
    \end{enumerate}
\end{definition}

\begin{remark}
    \label{rem:optimality}
    We complement the three optimality properties with explanatory remarks:
    \begin{enumerate}[(i)]
        \item 
            Criticality means that the fractional Hardy weight enjoying this property cannot be further improved by a pointwise bigger weight.
        \item 
            In place of non-attainability, the concept of \textit{null-criticality}, which can be viewed as an analogue of the combined properties of criticality and non-attainability, was employed in~\cite{keller-pinchover-pogorzelski_18_hardy_on_graphs,hake-keller-pogorzelski_26_frac_lap_graph}. The authors analyse whether the unique positive generalised eigenvector of the operator $(-\lap)^\alpha - \rho$, called the \textit{Agmon ground state}, belongs to the weighted sequence space $\ell^2(\N,\rho)$.
        \item \label{it:rem:optimality:3}
            It is easy to see that if a fractional Hardy weight exhibits optimality near infinity, then so does any other asymptotically equivalent fractional Hardy weight. It is equivalent to the equality
            \begin{equation}
                \label{eq:rem:opt_near_inf}
                \inf_{\substack{0\neq u\in C_0(\N) \\ u_n=0,\,\forall n<M}} \frac{\sum_{n=1}^\infty\overline{u_n}(-\lap)^\alpha u_n}{\sum_{n=1}^\infty \rho_n\abs{u_n}^2}
                  =1
                  \quad\text{ for all }M\ge 1.
            \end{equation}
    \end{enumerate}
\end{remark}

Our primary result is a fully explicit and optimal discrete fractional Hardy weight.

\begin{theorem}
    \label{thm:opt_frac_hardy}
    Let $\alpha>0$. Then the sequence $\rho^{(\alpha)}$, given by the formula
    \begin{equation}
        \label{eq:thm:opt_frac_hardy}
         \boxed{
          \rho^{(\alpha)}_n
          \coloneq \frac{4^{\alpha}\Gamma^2(\alpha+1/2)}{\pi}\frac{\Gamma(2n-1)}{\Gamma(2n+2\alpha-1)}
          \quad\text{ for }n\in\N,
          }
    \end{equation}
    is an optimal discrete fractional Hardy weight. 
\end{theorem}

\begin{remark}
    If $\alpha=\ell\in\N$ is an integer, the formula simplifies to
    \[
        \rho^{(\ell)}_n
          = \frac{\Gamma^2(\ell+1/2)}{\pi}\frac{1}{(n-1/2)n(n+1/2)\cdots(n+\ell-1)}
          \quad\text{ for }n\in\N,
    \]
    which reproduces the results of \eqref{eq:stamp_wacl_opt_disc_birm}. Moreover, it is noteworthy that if $\alpha=\ell+1/2\in\N_0+1/2$ is a half-integer, the expression also simplifies, and we obtain the inequality
    \[
        \sum_{n=1}^\infty\overline{u_n}(-\lap)^{\ell+1/2} u_n
          \ge \frac{(\ell!)^2}{\pi} \sum_{n=1}^\infty \frac{\abs{u_n}^2}{(n-1/2)n(n+1/2)\cdots(n+\ell-1/2)}
    \]
    for any $\ell\in\N_0$ and $u\in C_0(\N)$.
\end{remark}

With the aid of the Stirling formula, one deduces the asymptotic behaviour
\begin{equation}
    \label{eq:rho_a_asy}
    \rho^{(\alpha)}_n 
      =\frac{\Gamma^2(\alpha+1/2)}{\pi} \frac{1}{n^{2\alpha}}+\mathcal{O}\!\left(\frac{1}{n^{2\alpha+1}}\right),
      \quad\text{ for }n\to\infty.
\end{equation}
This motivates us to establish a lower bound for the optimal weight using only the leading term from~\eqref{eq:rho_a_asy}. This is our second main result: a discrete analogue of inequality \eqref{eq:cont_frac_hardy_on_R_+}.

\begin{theorem}[classical fractional Hardy inequality]
    \label{thm:clas_frac_hardy}
    Let $\alpha>0$. Then for any $u\in\ell^{2}(\N)$, with $u_n=0$ for all $n<\ceil{\alpha}$, the inequality
    \begin{equation}
        \label{eq:thm:clas_frac_hardy}
        \sum_{n=\ceil{\alpha}}^\infty \overline{u_n} (-\lap)^\alpha u_n
          \ge \frac{\Gamma^2(\alpha+1/2)}{\pi} \sum_{n=\ceil{\alpha}}^\infty \frac{\abs{u_n}^2}{n^{2\alpha}}
    \end{equation}
    holds and the constant $\Gamma^{2}(\alpha+1/2)/\pi$ is sharp. Moreover, the inequality \eqref{eq:thm:clas_frac_hardy} is strict unless $u\equiv0$.
\end{theorem}

\begin{remark}
    For $\alpha\in\N$, the inequality~\eqref{eq:thm:clas_frac_hardy} coincides with the discrete Birman inequality proved in \cite[Thm.~1.1]{huang-ye_24-hardy}; see also~\cite{bir_61} for the original continuous Birman inequality. 
\end{remark}

\begin{remark}
    The support restriction $u_n=0$ for all $n<\ceil{\alpha}$ and the shift in the summation index in \eqref{eq:thm:clas_frac_hardy} is necessary, since the inequality
    \[
        \rho^{(\alpha)}_n\ge\frac{\Gamma^2(\alpha+1/2)}{\pi}\frac{1}{n^{2\alpha}}
    \]
    fails in general. In fact, the reverse inequality holds for all $n\in\N$ if $\alpha\ge2$,
    precluding such a lower bound due to the criticality of $\rho^{(\alpha)}$.
\end{remark}

\section{Preliminaries}
\label{sec:prelim}

\subsection{Notation}

We use the notation $\Z$, $\N$, and $\N_0$ for the sets of integers, positive integers, and non-negative integers, respectively. 

We denote by $C(\Z)$ the space of complex sequences indexed by $\Z$, and by $C_0(\Z)$ its subspace of compactly (finitely) supported sequences. Similar notation is used for spaces $C(\N)$ and $C_0(\N)$ with the following \emph{zero-extension convention}. Whenever convenient, we embed $C(\N)$ into $C(\Z)$ extending the semi-infinite sequences indexed by $\N$ by zeros to non-positive indices, and similarly for spaces $C_0(\N)$ and $C_0(\Z)$.

The Hilbert spaces $\ell^2(\N)$ and $\ell^2(\Z)$ of square-summable sequences indexed by $\N$ and $\Z$, respectively, are considered to be endowed with the Euclidean inner products, anti-linear in the first argument.

Further, we will make use of discrete differential operators acting on $C(\Z)$: the \textit{discrete gradient} and \textit{discrete divergence} defined as
\begin{equation*} 
    \grad u_n := u_n - u_{n-1}
    \quad\text{ and }\quad
    \ddiv u_n := u_{n+1} - u_n
\end{equation*}
for all $n\in\Z$. Their composition defines the \textit{discrete Laplacian} $\lap_{\Z}\coloneq\ddiv\grad=\grad\ddiv$ on $\Z$, i.e.
\[
    \lap_{\Z} u_n=u_{n-1}-2u_n+u_{n+1}
    \quad\text{ for }n\in\Z.
\]

If not stated otherwise, we always assume that $\alpha>0$. We denote by $\ceil{\alpha}$ ($\floor{\alpha}$) the smallest (largest) integer greater (smaller) than or equal to $\alpha$, and define the \textit{fractional part} of $\alpha$ as
\begin{equation*}
    \{\alpha\}\coloneq \alpha - \ceil{\alpha} + 1\in(0,1].
\end{equation*}
Here, we adopt the convention where integers are mapped to 1 instead of 0.

Finally, throughout the text, we use $C_{a}$ or $C_{a,b}$ to denote generic positive constants depending only on parameters $a$ or $a,b$, respectively. The exact value of such a constant may change from line to line within derivations.

\subsection{Positive powers of the discrete Laplacian}

Recall that the discrete Laplace operator $-\lap_{\Z}$ is a self-adjoint Laurent operator on $\ell^{2}(\Z)$, which is diagonalised by the discrete Fourier transform. Concretely, we have
\[
\mathcal{F}(-\lap_{\Z})\mathcal{F}^{-1}=M_{2-2\cos\xi},
\]
where
\[
 \mathcal{F}u(\xi)\coloneq\frac{1}{\sqrt{2\pi}}\sum_{n\in\Z}\ee^{-\ii n\xi}\,u_{n} \quad\mbox{ in } L^{2}(-\pi,\pi),
\]
and $M_{2-2\cos\xi}$ denotes the multiplication operator by the function $2-2\cos\xi$ in $L^{2}(-\pi,\pi)$. Consequently, one defines the positive power of the discrete Laplace operator on $\ell^{2}(\Z)$ by
\begin{equation}
    \label{eq:def:disc_frac_lap}
    (-\lap_{\Z})^\alpha \coloneq \mathcal{F}^{-1} M_{(2 - 2 \cos \xi)^\alpha} \mathcal{F}
\end{equation}
for all $\alpha>0$. It follows that the spectrum of $(-\lap_{\Z})^\alpha$ is absolutely continuous, filling the interval $[0,4^{\alpha}]$, in particular $\|(-\lap_{\Z})^{\alpha}\|=4^{\alpha}$. Equivalent definitions can be found in~\cite{cia-ron_18, kel-nit_23}.

First, we compute the matrix entries of $(-\lap_{\Z})^{\alpha}$. We will make use of the generalised binomial coefficients,
\begin{equation}
    \label{eq:def_binom}
    \binom{a}{b} \!\coloneq \frac{\Gamma(a + 1)}{\Gamma(b + 1)\Gamma(a - b + 1)}
    \quad\text{ for } a>-1 \text{ and }b\in\R.
\end{equation}
(Recall that the reciprocal Gamma function is an entire function vanishing at non-positive integers.)

\begin{lemma}
    \label{lem:matrix_elements}
    Let $\alpha>0$. For all $m,n\in\Z$, we have
 	\[
        (-\lap_{\Z})^\alpha_{m,n} = (-1)^{m+n} \binom{2\alpha}{\alpha + m - n}.
	\]
\end{lemma}
\begin{proof}
    By definition~\eqref{eq:def:disc_frac_lap} and unitarity of the discrete Fourier transform $\mathcal{F}$, we have
    \[
        (-\lap_{\Z})^\alpha_{m,n}=\frac{1}{2\pi}\!\int_{-\pi}^\pi \ee^{\ii(n-m)x}\,(2-2\cos x)^\alpha\dd x.
    \]
    Since $2-2\cos{x}=4\sin^2{(x/2)}$ is even, the integral representation further simplifies to
    \[
        (-\lap_{\Z})^\alpha_{m,n} = 2\frac{4^\alpha}{\pi} \int_0^{\pi/2} \sin^{2\alpha}(t) \cos(2(m-n)t) \dd t.
    \]
The last integral may be computed with the aid of a slight modification of the identity~\cite[\S~3.631, Eq.~(8)]{grad-ryz_05-table}, which reads
    \[
        \int_0^{\pi/2} \sin^{\nu-1}(t) \cos(2kt) \dd t
          = \frac{(-1)^{k}\,\pi\,\Gamma(\nu)}{2^\nu\,\Gamma\!\paren{\frac{\nu+1}{2}+k}\Gamma\!\paren{\frac{\nu+1}{2}-k}}
          \quad\text{ for }k\in\Z \text{ and }\Re{\nu}>0.
    \]
    Altogether, in view of definition~\eqref{eq:def_binom}, we obtain the desired result
    \[
        (-\lap_{\Z})^\alpha_{m,n} 
          = (-1)^{m+n}\frac{\Gamma(2\alpha+1)}{\Gamma(\alpha+1+m-n)\Gamma(\alpha+1-m+n)}
          = (-1)^{m+n} \binom{2\alpha}{\alpha+m-n}.
    \qedhere
    \]
\end{proof}

\begin{remark}
   Since $-\lap_{\Z}$ is a self-adjoint Laurent operator, its positive power $(-\lap_{\Z})^{\alpha}$ must possess the same properties. This can be checked directly by observing that the matrix element formula of Lemma~\ref{lem:matrix_elements} depends only on $|m-n|$. Note also that if $\alpha\in\N$, the Laurent matrix of $(-\lap_{\Z})^{\alpha}$ is banded.
\end{remark}

Let $P_{+}:\ell^{2}(\N)\to\ell^{2}(\Z)$ be the embedding map of $\ell^{2}(\N)\hookrightarrow\ell^{2}(\Z)$, which extends a semi-infinite sequence of $\ell^{2}(\N)$ by zeros to non-positive indices. Then we define the \emph{fractional discrete Laplacian on the half-line} $\N$ as the compression
\begin{equation}
 (-\lap)^{\alpha}:=P_{+}^{*}(-\lap_{\Z})^{\alpha}P_{+}.
\label{eq:def-frac-lap_abstract}
\end{equation}
It follows, using Lemma~\ref{lem:matrix_elements}, that $(-\lap)^{\alpha}$ is a bounded Toeplitz operator on $\ell^{2}(\N)$, whose matrix elements can be expressed in several equivalent forms:
\begin{align}
 (-\lap)^\alpha_{m,n} = (-1)^{m+n}\, \binom{2\alpha}{\alpha + m - n}
 &=\frac{(-1)^{m+n}\,\Gamma(2\alpha+1)}{\Gamma(\alpha+1+m-n)\Gamma(\alpha+1-m+n)} \label{eq:def-frac-lap_matrix1}\\
 &=-\frac{\Gamma(2\alpha+1)\sin{(\pi\alpha)}}{\pi}\,\frac{\Gamma(|m-n|-\alpha)}{\Gamma(|m-n|+\alpha+1)} \label{eq:def-frac-lap_matrix2}
\end{align}
for all $m,n\in\N$, where \eqref{eq:def-frac-lap_matrix2} follows from the Euler reflection formula \cite[Eq.~(5.5.3)]{DLMF}
\begin{equation}\label{eq:euler_reflection}
        \Gamma(z)\Gamma(1-z)=\frac{\pi}{\sin(\pi z)},
\end{equation}
and is to be interpreted as the respective limit if $\alpha\in\N$.

\begin{remark}
The operator $(-\lap)^{\alpha}$ should not be confused with the operator $(-\lap_{\N})^{\alpha}$ studied in~\cite{gerhat-krejcirik-stampach_23_criticality}. Matrix entries of the latter read
    \begin{equation}
        \label{eq:mat_el_frac_lap_on_N}
        (-\lap_\N)^\alpha_{m,n}=(-1)^{m+n} \left[ \binom{2\alpha}{\alpha+m-n} - \binom{2\alpha}{\alpha+m+n} \right]
        \quad\text{ for }m,n\in\N.
    \end{equation}
The matrix of $(-\lap_\N)^\alpha$ is not Toeplitz and differs from the matrix of $(-\lap)^\alpha$ by a Hankel matrix. Due to this difference, the subcritical range for $(-\lap_\N)^\alpha$ is restricted to exponents $\alpha\in(0,3/2)$, while $(-\lap)^\alpha$ is subcritical for all $\alpha>0$. Moreover, the Toeplitz structure of $(-\lap)^\alpha$ is essential for deriving fractional Hardy inequalities for $\alpha>1$.
\end{remark}

The matrix multiplication action of $(-\lap)^{\alpha}$ on a semi-infinite column vector $v$, namely 
\begin{equation}
 (-\lap)^{\alpha}v_{n}=\sum_{m=1}^{\infty}(-\lap)^{\alpha}_{n,m}v_{m}, \quad n\in\N,
\label{eq:def-frac-lap_extended}
\end{equation}
makes sense on sequence spaces larger than $\ell^{2}(\N)$. Requiring absolute convergence of the series, we introduce the domain
\begin{equation}
    \label{eq:def:dom_a}
    \mathcal{D}_\alpha 
      \coloneq \Bigl\{ v\in C(\N) \Bigm\vert \sum_{m=1}^\infty \bigl|(-\lap)^\alpha_{n,m} v_m\bigr| < \infty,\;\forall n\in\N \Bigr\}.
\end{equation}
Observe that the asymptotic expansion for the ratio of Gamma functions~\cite[Eq.~(5.11.13)]{DLMF}
\begin{equation}
    \label{eq:Gamma_ratio_exp}
    \frac{\Gamma(x+a)}{\Gamma(x+b)}=x^{a-b}\left[1+\mathcal{O}\!\left(\frac{1}{x}\right)\!\right],
    \quad\text{ as } x\to\infty,
    \text{ for } a,b\in\R,
\end{equation}
in conjunction with the formula~\eqref{eq:def-frac-lap_matrix2}, yields the estimate
\begin{equation}
    \label{eq:estimate_frac_lap}
    \abs{(-\lap)^\alpha_{m,n}}
    \le \frac{C_\alpha}{\abs{m-n}^{2\alpha+1}}
    \quad \text{ for } m\neq n.
\end{equation}
In particular, this implies that $\ell^2(\N)\subset\ell^\infty(\N)\subset \mathcal{D}_\alpha$.

\begin{remark}
    In~\cite{das-fernandez_26-frac_hardy}, the authors exploit the fact that the fractional Laplacian $(-\lap_\N)^\alpha$, determined by~\eqref{eq:mat_el_frac_lap_on_N}, can be interpreted as a formal Schrödinger operator on a suitable nonlocal graph for $\alpha \in (0,1]$. This observation allows them to apply the abstract framework of optimal Hardy inequalities on graphs developed in~\cite{keller-pinchover-pogorzelski_18_hardy_on_graphs}. For completeness, we demonstrate that the operator $(-\lap)^\alpha$ likewise induces a formal Schrödinger operator for $\alpha\in(0,1]$ in the following sense. A \textit{graph} over a discrete set $X$ is a symmetric function $e:X\times X\to[0,\infty)$, such that
    \[
        e(x,x)=0
        \quad\text{ and }\quad
        \sum_{y\in X} e(x,y)<\infty
        \quad\text{ for all }x\in X.
    \]
    If $e(x,y)>0$, we say that $x,y\in X$ are \textit{connected by an edge} and write $x\sim y$. The \textit{formal Laplacian}~$L$, acting on the space
    \[
        \mathcal{D}(X)\coloneq\Bigl\{ u:X\to\C \Bigm\vert \sum_{y\sim x}e(x,y)\abs{u(y)}<\infty, \forall x\in X\Bigr\},
    \]
    is defined by the formula
    \[
        L u(x)
          \coloneq\sum_{y\sim x}e(x,y)\bigl(u(x)-u(y)\bigr)
        \quad\text{ for all } x\in X.
    \]
    Finally, the \textit{formal Schrödinger operator $H$ associated with a potential} $q:X\to\R$ is defined on $\mathcal{D}(X)$ by $H\coloneq L+q$, i.e. $Hu(x)\coloneq Lu(x)+q(x)u(x)$ for all $x\in X$ and $u\in\mathcal{D}(X)$.
    
    We now set $X \coloneq \N$, define $e(m,n) \coloneq -(-\lap)_{m,n}^\alpha$ for $m \neq n$, and $e(n,n) \coloneq 0$ for $n \in \N$. The resulting graph is well-defined for every $\alpha \in (0,1]$, since $e$ is symmetric and non-negative, as one easily verifies by inspection of formula~\eqref{eq:def-frac-lap_matrix1}. Moreover, we introduce the potential
    \begin{equation}
        \label{eq:potential_expl}
        q_\alpha(n) 
          \coloneq \sum_{m\in\N} (-\lap)_{m,n}^\alpha
          =(-1)^{n+1}\binom{2\alpha-1}{\alpha-n},
        \quad \text{ for } n \in \N
        \text{ and }\alpha>0.
    \end{equation}
    The explicit form can be computed by invoking the formula \eqref{eq:def-frac-lap_matrix1} and rewriting the sum as a \textit{telescoping sum} using Pascal's rule for binomials. Consequently, we may write 
    \begin{equation}
        \label{eq:form_SO}
        (-\lap)^\alpha u(n)
          \coloneq - \sum_{m=1}^\infty (-\lap)_{n,m}^\alpha \bigl( u(n) - u(m) \bigr)
          + u(n) \sum_{m=1}^\infty (-\lap)_{n,m}^\alpha
          \equiv L u(n) + q_\alpha(n)u(n),
    \end{equation}
    and thus the operator $(-\lap)^\alpha$ indeed induces a formal Schrödinger operator for all $\alpha\in(0,1]$. 
    
	Although this fact will not be used explicitly, it constitutes the main idea behind the \textit{ground state transform} (cf.~\cite[Proposition 2.4]{keller-pinchover-pogorzelski_18_hardy_on_graphs}), given by the Hardy-type identity in Theorem~\ref{thm:frac_hardy_eq} below. In contrast, if $\alpha>1$, neither $(-\lap)^\alpha$ nor $(-\lap_\N)^\alpha$ can be viewed as such operators, since $e(m,n)$ is no longer non-negative for all $m,n\in\N$. This is the main obstacle which prevents the method used in \cite{das-fernandez_26-frac_hardy} from covering the case $\alpha\in(1,3/2)$. As we do not rely on the theory requiring non-negative kernels, we overcome this lack of non-negativity by demonstrating Theorem~\ref{thm:opt_frac_hardy} directly.
\end{remark}

\section{Abstract construction of fractional Hardy weights}
\label{sec:abstract_results}

In this section, we explain how to derive abstract fractional Hardy inequalities given a \textit{parameter sequence} $\g$ that satisfies certain sufficient conditions. The key idea is to establish fractional Hardy identities by explicitly characterising the remainder terms appearing in the corresponding inequalities. Furthermore, both the weight and the remainders will be represented explicitly in terms of $\g$; a feature that will play a crucial role in verifying the optimality of the corresponding weights.

For this purpose, we begin by formulating the following auxiliary lemma, which, for $\alpha\in(0,1]$, identifies the \textit{quadratic form} associated with the formal Schrödinger operator~\eqref{eq:form_SO} corresponding to the fractional Laplacian $(-\lap)^\alpha$.

\begin{lemma}
    \label{lem:lap_quad_formula}
    Let $\alpha>0$. Then for any $u\in C_0(\N)$, we have the identity
    \begin{equation*}
        \sum_{n=1}^\infty  \overline{u_n}(-\lap)^\alpha u_n
          =-\frac{1}{2}\sum_{m,n\in\N}(-\lap)^\alpha_{n,m}\abs{u_m-u_n}^2 - \sum_{n=1}^\infty  (-1)^n \binom{2\alpha-1}{\alpha-n} \abs{u_n}^2.
    \end{equation*}
\end{lemma}
\begin{proof}
    The quadratic form on the left-hand side can be rewritten as
    \[
        \sum_{n=1}^\infty  \overline{u_n}(-\lap)^\alpha u_n
          = \sum_{n=1}^\infty  \!\paren{\sum_{m=1}^\infty(-\lap)^\alpha_{n,m}}\!\abs{u_n}^2 - \! \sum_{m,n\in\N}(-\lap)^\alpha_{n,m}\overline{u_n}(u_n-u_m).
    \]
    To the first term on the right, we apply identity~\eqref{eq:potential_expl}. Moreover, the second term is rewritten as
    \[
            \sum_{m,n\in\N}(-\lap)^\alpha_{n,m}\overline{u_n}(u_n-u_m)
              =\sum_{m,n\in\N}(-\lap)^\alpha_{n,m}\overline{u_m}(u_m-u_n)
              =\frac{1}{2}\sum_{m,n\in\N}(-\lap)^\alpha_{n,m}\abs{u_m-u_n}^2
    \]
    by the symmetry of the operator $(-\lap)^\alpha$. The claim follows.
\end{proof}

We are now in position to derive fractional Hardy identities for any $\alpha>0$, analogous to the \textit{ground state representation}~\cite[Proposition 3.4]{das-fernandez_26-frac_hardy} of the operator $(-\lap_\N)^\alpha$ for $\alpha\in(0,1]$.

\begin{theorem}
    \label{thm:frac_hardy_eq}
    Let $\alpha>0$ and $\g\in\mathcal{D}_\alpha$ such that $\g_n>0$ for all $n\ge1$. Then for any $u\in C_0(\N)$, we have the identity
    \begin{equation}
        \label{eq:thm:frac_hardy_eq}
        \sum_{n=1}^\infty  \overline{u_n} (-\lap)^\alpha u_n - \sum_{n=1}^\infty \frac{(-\lap)^\alpha \g_n}{\g_n} \abs{u_n}^2
          = -\!\sum_{1\le n<m} \!(-\lap)^\alpha_{m,n}\,\g_m\,\g_n\abs{\frac{u_m}{\g_m}-\frac{u_n}{\g_n}}^2.
    \end{equation}
\end{theorem}
\begin{proof}
    For any $m\neq n$, elementary computations yield
    \begin{align*}
        \g_m\,\g_n\abs{\frac{u_m}{\g_m}-\frac{u_n}{\g_n}}^2 &=\frac{\g_n}{\g_m}\abs{u_m}^2+\frac{\g_m}{\g_n}\abs{u_n}^2 -2\Re{(\overline{u_n}u_m)} \\
            &= \abs{u_m-u_n}^2 +\frac{\g_n}{\g_m}\abs{u_m}^2+\frac{\g_m}{\g_n}\abs{u_n}^2 - \abs{u_n}^2-\abs{u_m}^2. 
    \end{align*}
    Multiplying both sides by $(-\lap)^\alpha_{m,n}/2$ and summing over all $m,n\in\N$, we get
    \begin{align*}
        \sum_{1\le n<m} \!\!(-\lap)^\alpha_{m,n}\,\g_m\,\g_n\abs{\frac{u_m}{\g_m}-\frac{u_n}{\g_n}}^2
          =\; &\frac{1}{2}\sum_{m,n\in\N}(-\lap)^\alpha_{n,m}\abs{u_m-u_n}^2
             +\sum_{n=1}^\infty  (-1)^n \binom{2\alpha-1}{\alpha-n} \abs{u_n}^2 \\
            &+ \sum_{n=1}^\infty \frac{\abs{u_n}^2}{\g_n}\paren{\sum_{m=1}^\infty (-\lap)^\alpha_{m,n}\g_m}\!,
    \end{align*}
    where we used the symmetry of $(-\lap)^\alpha$ and formula~\eqref{eq:potential_expl}. Applying Lemma~\ref{lem:lap_quad_formula} to the first line of the right-hand side and employing definition~\eqref{eq:def-frac-lap_extended} in the second, we obtain~\eqref{eq:thm:frac_hardy_eq}.
\end{proof}

The next step is to ensure the non-negativity of the remainder term appearing on the right-hand side of the fractional Hardy identity~\eqref{eq:thm:frac_hardy_eq}. However, it follows from formula~\eqref{eq:def-frac-lap_matrix1} that $(-\lap)^\alpha_{m,n}\le 0$ is true for all $m\neq n$ only when $\alpha\in(0,1]$. To overcome this limitation, we combine Theorem~\ref{thm:frac_hardy_eq} for $\alpha\in(0,1]$ with an iteration of the following weighted Hardy-type identity, which appeared in~\cite[Theorem 0]{stampach-waclawek_24-birman}. Its proof, which is analogous to the argument used in the proof of Theorem~\ref{thm:frac_hardy_eq}, can be found therein.
\begin{lemma}
    \label{lem:weighted_hardy}
    Let $V,\g\in C(\N)$ such that $\g_{n}>0$ for all $n\in\N$. Then for all $u\in C_0(\N)$, we have the identity
    \begin{equation}
        \label{eq:lem:weighted_hardy}
        \sum_{n=1}^\infty  V_{n} \abs{\grad u_{n}}^{2} + \sum_{n=1}^\infty  \frac{\ddiv(V\grad \g)_{n}}{\g_n}\abs{u_{n}}^{2} 
          = \sum_{n=1}^\infty  V_{n+1} \,\g_{n+1}\,\g_n\abs{\frac{u_{n+1}}{\g_{n+1}}-\frac{u_{n}}{\g_{n}}}^2.
    \end{equation}
\end{lemma}

The iterative procedure relies on an appropriate choice of the parameter sequence $\g$ at each step, which enables us to derive tractable representations of the weights $V$ as explicit functions of $\g$. However, this imposes additional monotonicity requirements on $\g$. The resulting fractional Hardy identity, which is formulated in the next theorem, may be of independent interest.

\begin{theorem}
    \label{thm:frac_birman_eq}
    Let $\alpha>0$. Suppose that $\g\in C(\N)$ satisfies the assumption
    \begin{equation}
        \label{assump:A1}
        \tag{A1}
        \grad^k\g\in\mathcal{D}_{\alpha-k}
        \quad\text{ with }\quad
        \grad^k\g_n>0
        \quad\text{ for all } n\in\N \text{ and } k\in\{0,\dots,\ceil{\alpha}-1\}.
    \end{equation}
    Then for all $u\in C_0(\N)$, we have the identity
    \begin{equation}
        \label{eq:thm:frac_birman_eq}
        \sum_{n=1}^\infty  \overline{u_n} (-\lap)^\alpha u_n - \sum_{n=1}^\infty  \frac{(-\lap)^\alpha \g_n}{\g_n} \abs{u_n}^2
          =\sum_{k=1}^{\ceil{\alpha}} R_k^{(\alpha)} (\g,u), 
    \end{equation}
    where the remainders are given by formulas
    \begin{equation}
        \label{eq:thm:frac_birman_rem}
        R_k^{(\alpha)}(\g,u) \coloneq
        \begin{cases}
          \;\;\;\;\sum\limits_{n=1}^\infty \frac{(-\lap)^{\alpha-k}\grad^{k}\g_{n+1}}{\grad^{k}\g_{n+1}} 
            \scalebox{0.7}{$\grad^{k-1}\g_{n+1} \grad^{k-1}\g_{n}$} \abs{ \frac{\grad^{k-1}u_{n+1}}{\grad^{k-1}\g_{n+1}} - \frac{\grad^{k-1}u_{n}}{\grad^{k-1}\g_{n}}}^{2}
            & \text{ if }k<\ceil{\alpha}, \\
          -\!\!\!\sum\limits_{1\le n<m} \!\! \scalebox{0.8}{$(-\lap)^{\!\{\alpha\}}_{m,n}$} \,
            \scalebox{0.7}{$\grad^{\ceil{\alpha}-1}\g_{m} \grad^{\ceil{\alpha}-1}\g_{n}$} \abs{ \frac{\grad^{\ceil{\alpha}-1}u_{m}}{\grad^{\ceil{\alpha}-1}\g_{m}} - \frac{\grad^{\ceil{\alpha}-1}u_{n}}{\grad^{\ceil{\alpha}-1}\g_{n}}}^{2}
            & \text{ if }k=\ceil{\alpha}.
        \end{cases}
    \end{equation}
\end{theorem}
\begin{proof}
    The proof proceeds by induction on $\ceil{\alpha}\in\N$. The base case $\ceil{\alpha}=1$ coincides with Theorem~\ref{thm:frac_hardy_eq}.

    Suppose $\ceil{\alpha}\ge2$. Since $(-\lap_\Z)^\alpha$ is a Laurent operator, it commutes with the forward shift operator, and hence also with the discrete divergence. Recalling that $\lap_{\Z}=\ddiv\grad$ and $\ddiv^*=-\grad$ in $\ell^{2}(\Z)$, one readily verifies, for all $u\in C_0(\N)$, the equality
    \[
        \sum_{n=1}^\infty  \overline{u_n}(-\lap)^\alpha u_n = \langle u,(-\lap_{\Z})^{\alpha}u\rangle_{\ell^{2}(\Z)} = 
        \langle \grad u,(-\lap_{\Z})^{\alpha-1}\grad u\rangle_{\ell^{2}(\Z)} = 
        \sum_{n=1}^\infty  \overline{\grad u_n} (-\lap)^{\alpha-1}\grad u_n.
    \]
Let $\g\in C(\N)$ satisfy assumption~\eqref{assump:A1}. Then $\grad\g\in C(\N)$ fulfils the assumption for $\alpha-1$, and therefore the induction hypothesis, applied to $\grad u\in C_0(\N)$, with $\grad\g$, yields
    \begin{equation}
        \label{eq:pr:thm:frac_hardy_eq_++:1}
        \sum_{n=1}^\infty \overline{u_n}(-\lap)^\alpha u_n
          = \sum_{n=1}^\infty  \frac{(-\lap)^{\alpha-1}\grad\g_n}{\grad\g_n}\abs{\grad u_n}^2 + \sum_{k=1}^{\ceil{\alpha}-1} R_k^{(\alpha-1)}(\grad\g,\grad u).
    \end{equation}
    In view of the defining formula~\eqref{eq:thm:frac_birman_rem} of the remainders on the right-hand side, the obvious identity $\{\alpha\}=\{\alpha-1\}$ implies, for all $k\in\{1,\dots,\ceil{\alpha}-1\}$, that
    \[
        R_k^{(\alpha-1)}(\grad\g,\grad u) = R_{k+1}^{(\alpha)}(\g,u),
        \quad\text{ thus }\quad
        \sum_{k=1}^{\ceil{\alpha}-1} R_k^{(\alpha-1)}(\grad\g,\grad u) = \sum_{k=2}^{\ceil{\alpha}} R_k^{(\alpha)}(\g,u).
    \]

    Next, invoking Lemma~\ref{lem:weighted_hardy} with the weight $V_n\coloneq(-\lap)^{\alpha-1}\grad\g_n/\grad\g_n$, for all $n\in\N$, the first term on the right-hand side of~\eqref{eq:pr:thm:frac_hardy_eq_++:1} rewrites as
    \[
        \sum_{n=1}^\infty  \frac{(-\lap)^{\alpha-1}\grad\g_n}{\grad\g_n}\abs{\grad u_n}^2
          = -\sum_{n=1}^\infty  \frac{\ddiv(V\grad\g)_n}{\g_n} \abs{u_n}^2
            +\sum_{n=1}^\infty  V_{n+1}\,\g_{n+1}\,\g_n \abs{\frac{u_{n+1}}{\g_{n+1}}-\frac{u_{n}}{\g_{n}}}^2.
    \]
    While the divergence expression on the right-hand side simplifies as
    \[
        -\ddiv(V\grad\g) = -\ddiv\paren{\frac{(-\lap)^{\alpha-1}\grad\g}{\grad\g}\grad\g} = (-\lap)^\alpha\g,
    \]
    the second sum is equal to the remainder $R_1^{(\alpha)}(\g,u)$ by definition~\eqref{eq:thm:frac_birman_rem}. Therefore
    \[
        \sum_{n=1}^\infty  \frac{(-\lap)^{\alpha-1}\grad\g_n}{\grad\g_n}\abs{\grad u_n}^2 = \sum_{n=1}^\infty  \frac{(-\lap)^\alpha\g_n}{\g_n}\abs{u_n}^2+R_1^{(\alpha)}(\g,u),
    \]
    which, in combination with the identity~\eqref{eq:pr:thm:frac_hardy_eq_++:1}, completes the proof of Theorem~\ref{thm:frac_birman_eq}.
\end{proof}

Now we impose further constraints on the parameter sequence $\g$, which guarantee non-negativity of the new remainder terms in~\eqref{eq:thm:frac_birman_eq}, getting an abstract fractional Hardy inequality formulated in the next theorem. Its proof follows immediately from Theorem~\ref{thm:frac_birman_eq}.

\begin{theorem}
    \label{thm:abs_frac_birman_ineq}
    Let $\alpha>0$. Suppose that $\g\in C(\N)$ satisfies assumptions~\eqref{assump:A1}, and, in addition
    \begin{equation}
        \label{assump:A2}
        \tag{A2}
        (-\lap)^{\alpha-k}\grad^k\g_n\ge 0
        \quad\text{ for all } n\ge2 \text{ and } k\in\{1,\dots,\ceil{\alpha}-1\}.
    \end{equation}
    Then for all $u\in C_0(\N)$, we have the inequality
    \begin{equation}
        \label{eq:thm:abs_frac_birman_ineq}
        \sum_{n=1}^\infty  \overline{u_n} (-\lap)^\alpha u_n \ge \sum_{n=1}^\infty  \rho_n(\g) \abs{u_n}^2,
    \end{equation}
    where $\rho(\g)\coloneq(-\lap)^\alpha\g/\g$. If moreover,
    \begin{equation}
        \label{assump:A3}
        \tag{A3}
        (-\lap)^\alpha\g_n\ge0
        \quad\text{ for all }n\ge1,
    \end{equation}
    then $\rho(\g)\ge0$ is a discrete fractional Hardy weight.
\end{theorem}
    
\section{Proofs of Theorems~\ref{thm:opt_frac_hardy} and~\ref{thm:clas_frac_hardy}}
\label{sec:proof_main_results}

This section proves Theorems \ref{thm:opt_frac_hardy} and \ref{thm:clas_frac_hardy} using the abstract results of Section~\ref{sec:abstract_results} with the concrete choice of the parameter sequence defined by
\begin{equation}
    \label{eq:def:frac_g}
    \g^{(\alpha)}_n \coloneq \frac{\Gamma(n+\alpha-1/2)}{\Gamma(n)},
    \quad\text{ for } n\in\N,
\end{equation}
with the convention $\g^{(\alpha)}_n\coloneq0$ if $n\le0$. This specific form allows for closed-form expressions when the difference operators $\grad$ and $(-\lap)^\alpha$ are applied. This is crucial for verifying assumptions~\eqref{assump:A1}--\eqref{assump:A3}, which, in view of Theorem~\ref{thm:abs_frac_birman_ineq}, imply that the sequence
\[
    \rho^{(\alpha)}_n 
      \coloneq \rho_n(\g^{(\alpha)})
      = \frac{(-\lap)^\alpha \g_n^{(\alpha)}}{\g_n^{(\alpha)}},
    \quad\text{ for } n\in\N,
\]
constitutes a fractional Hardy weight for every $\alpha>0$. In this way, in Section~\ref{subsec:pr:opt_ineq}, we demonstrate that the weight in \eqref{eq:thm:opt_frac_hardy} satisfies the discrete fractional Hardy inequality. To complete the proof of Theorem~\ref{thm:opt_frac_hardy}, we establish the optimality of the weight $\rho^{(\alpha)}$ in Section~\ref{subsec:pr:opt} by closely inspecting the remainders from identity~\eqref{eq:thm:frac_birman_eq}.
Lastly, Section~\ref{subsec:pr:clas_ineq} derives a lower bound for the optimal weight $\rho^{(\alpha)}$ and proves Theorem~\ref{thm:clas_frac_hardy}.

\subsection{Proof of the fractional Hardy inequality for \texorpdfstring{$\rho^{(\alpha)}$}{rho}}
\label{subsec:pr:opt_ineq}

We begin by showing that the sequence $\g^{(\alpha)}$, defined by~\eqref{eq:def:frac_g}, behaves nicely under the action of the discrete derivative. The property mimics the elementary differentiation formula for monomials
\[
    \frac{\dd^k x^{\alpha-1/2}}{\dd x^k}=\pochhammer{\alpha-k+1/2}{k}\, x^{\alpha-k-1/2},
\]
where $(\nu)_n\coloneq \nu(\nu+1)\cdots(\nu+n-1)$ denotes the Pochhammer symbol for $\nu\in\R$.

\begin{lemma}
    \label{lem:frac_grad_g}
    Let $\alpha>0$ and $k\in\{0,\dots,\ceil{\alpha}-1\}$. Then, for all $n\in\Z$, we have the identity
    \begin{equation}
        \label{eq:lem:frac:grad_g}
        \grad^k \g^{(\alpha)}_n = \pochhammer{\alpha-k+1/2}{k} \g_n^{(\alpha-k)}.
    \end{equation}
\end{lemma}
\begin{proof}
    It is sufficient to verify the formula for $\alpha\ge1$ and $k=1$. The general result then follows immediately by mathematical induction. Identity~\eqref{eq:lem:frac:grad_g} holds trivially if $n\le0$. On the other hand, for $n\ge1$, we employ the recurrence relation for the Gamma function to obtain
    \[
        \grad\g_n^{(\alpha)}
          = \frac{\Gamma(n+\alpha-1/2)}{\Gamma(n)}\paren{1-\frac{n-1}{n+\alpha-3/2}}\!
          =(\alpha-1/2)\g_n^{(\alpha-1)},
    \]
    which coincides with the desired identity.
\end{proof}


The next lemma provides an explicit form of the sequence $(-\lap)^\alpha \g^{(\alpha)}$.

\begin{lemma}
    \label{lem:frac_lap_g}
    Let $\alpha>0$. Then $\g^{(\alpha)}\in\mathcal{D}_\alpha$ and, for all $n\in\N$, we have the identity
    \begin{equation*}
        (-\lap)^\alpha\g_n^{(\alpha)} 
          = \frac{\Gamma^2(\alpha+1/2)}{\pi} \frac{\Gamma(n-1/2)}{\Gamma(n+\alpha)}.
    \end{equation*}
\end{lemma}
\begin{proof}
    Firstly, it follows from~\eqref{eq:Gamma_ratio_exp}, \eqref{eq:estimate_frac_lap}, and \eqref{eq:def:frac_g} that
    \[
        \abs{(-\lap)^\alpha_{n,m}\,\g^{(\alpha)}_m}
          \le \frac{C_{\alpha,n}}{m^{\alpha+3/2}}
    \]
    for all $m,n\in\N$. Thus $\g^{(\alpha)}\in\mathcal{D}_\alpha$; see \eqref{eq:def:dom_a}.
    
    Next, a straightforward computation shows that, for any $n\in\N$, we have
\begin{align*}
        (-\lap)^\alpha\g_n^{(\alpha)} &= \sum_{m=1}^\infty (-1)^{m+n} \binom{2\alpha}{\alpha+m-n} \frac{\Gamma(m+\alpha-1/2)}{\Gamma(m)} \\
          &= (-1)^{n+1} \Gamma(2\alpha+1) \sum_{m=0}^\infty \frac{(-1)^m}{m!} \frac{\Gamma(m+\alpha+1/2)}{\Gamma(\alpha+m-n+2)\Gamma(\alpha-m+n)} \\
     &= (-1)^{n+1} \Gamma(2\alpha+1) \frac{\Gamma(\alpha+1/2)}{\Gamma(\alpha+n)} \sum_{m=0}^\infty \frac{\pochhammer{\alpha+1/2}{m}\pochhammer{1-\alpha-n}{m}}{m!\,\Gamma(\alpha+m-n+2)}.
\end{align*}
    With the aid of \textit{Gauss's hypergeometric theorem}~\cite[Eq.~(15.4.20)]{DLMF}, which reads
    \[
        \sum_{m=0}^\infty \frac{\pochhammer{a}{m}\pochhammer{b}{m}}{m!\,\Gamma(c+m)} 
          = \frac{\Gamma(c-a-b)}{\Gamma(c-a)\Gamma(c-b)},
          \quad\text{ whenever }c-a-b>0,
    \]
    the expression further simplifies to
    \[
        (-\lap)^\alpha\g_n^{(\alpha)} =\Gamma^{2}(\alpha+1/2)\, \frac{(-1)^{n+1}}{\Gamma(\alpha+n)\Gamma(3/2-n)}.
    \]
 To conclude the proof, it suffices to apply the Euler reflection formula~\eqref{eq:euler_reflection}.
\end{proof}

\begin{proof}[Proof that $\rho^{(\alpha)}$ is a fractional Hardy weight]
    The assertion follows from the verification of assumptions~\eqref{assump:A1}--\eqref{assump:A3} in Theorem~\ref{thm:abs_frac_birman_ineq}, which, together with Lemma~\ref{lem:frac_lap_g}, subsequently imply that the weight, given by
    \[
        \rho^{(\alpha)}_n
          =\frac{(-\lap)^\alpha\g_n^{(\alpha)}}{\g^{(\alpha)}_n}
          =\frac{\Gamma^2(\alpha+1/2)}{\pi}\frac{\Gamma(n-1/2)\Gamma(n)}{\Gamma(n-1/2+\alpha)\Gamma(n+\alpha)}
    \]
    for $n\in\N$, is a discrete fractional Hardy weight. To arrive at the expression~\eqref{eq:thm:opt_frac_hardy}, we apply the duplication formula \cite[Eq.~(5.5.5)]{DLMF}
    \[
    \Gamma(2z)=\frac{2^{2z-1}}{\sqrt{\pi}}\Gamma(z)\Gamma\!\left(z+\frac{1}{2}\right).
    \]

    Since $\g^{(\alpha)}_n>0$ for all $n\ge1$ by the defining formula~\eqref{eq:def:frac_g}, Lemma~\ref{lem:frac_grad_g} immediately implies the inequalities in assumption~\eqref{assump:A1}. Moreover, by exploiting the transition between the parameter sequences $\g^{(\alpha)}$ provided by formula~\eqref{eq:lem:frac:grad_g}, we deduce that
    \[
        \grad^k\g^{(\alpha)}\in\mathcal{D}_{\alpha-k}
        \quad\text{ and }\quad
        (-\lap)^{\alpha-k}\grad^k\g^{(\alpha)}_n>0
        \quad\text{ for all }n\in\N \text{ and }k\in\{0,\dots,\ceil{\alpha}-1\},
    \]
    once this property has been verified for the case $k=0$ only. This, however, is the content of Lemma~\ref{lem:frac_lap_g}, and so the assumptions~\eqref{assump:A1}--\eqref{assump:A3} are all satisfied. The proof is complete.
\end{proof}

\begin{remark}
    Without going into technical details, we note that the one-parameter family
    \[
        \g^{(\alpha)}_n(s) 
          \coloneq \frac{\Gamma(n+\alpha-s)}{\Gamma(n)},
          \quad\text{ for } n\in\N,
    \]
    with the convention $\g^{(\alpha)}_n(s)\coloneq0$ for $n\le0$, also generates discrete fractional Hardy weights for any $s\in(0,1)$. The weights are given by the formula
    \[
        \rho^{(\alpha)}_n(s) 
          \coloneq \frac{(-\lap)^\alpha\g^{(\alpha)}_n(s)}{\g^{(\alpha)}_n(s)} 
          = \frac{\Gamma(s+\alpha)\Gamma(1-s+\alpha)}{\Gamma(s)\Gamma(1-s)} \frac{\Gamma(n-s)\Gamma(n)}{\Gamma(n-s+\alpha)\Gamma(n+\alpha)}
    \]
    for $n\in\N$.
    It can be shown that the leading term 
    \[
          \lim_{n\to\infty} n^{2\alpha}\rho^{(\alpha)}_n(s)=\frac{\Gamma(s+\alpha)\Gamma(1-s+\alpha)}{\Gamma(s)\Gamma(1-s)},
    \]
    as a function of $s$ on $(0,1)$, has a sharp maximum at $s=1/2$. Consequently, the weight can only be optimal near infinity if $s=1/2$, which corresponds to our choice in definition~\eqref{eq:def:frac_g}.
    
\end{remark}

\subsection{Proof of the optimality}
\label{subsec:pr:opt}
Here, we show the optimality of the fractional Hardy weight $\rho^{(\alpha)}$, and hence complete the proof of Theorem~\ref{thm:opt_frac_hardy}. The proof will be organised into three distinct parts, in which we demonstrate that the weight possesses all the optimality properties from Definition~\ref{def:optimality}: criticality, non-attainability, and optimality near infinity.

The leitmotif and central idea common to all three parts is that the remainders appearing on the right-hand side of identity~\eqref{eq:thm:frac_birman_eq} vanish under the formal substitution $u=\g$, i.e. $R_k^{(\alpha)}(\g,\g)=0$ for all $k\in\{1,\dots,\ceil{\alpha}\}$. Unfortunately, such a substitution cannot be performed directly, since no parameter sequence $\g$ satisfying \eqref{assump:A1} is compactly supported, i.e. $\g \notin C_0(\N)$. For this reason, we approximate by suitably chosen sequences $\{u^N\}_{N\ge2}\subset C_0(\N)$ such that $u^N\to\g$ pointwise, as $N\to\infty$, and inspect the asymptotic behaviour of $R_k^{(\alpha)}(\g,u^N)$ for $N\to\infty$.

In particular, for the proof of criticality and optimality near infinity, we select a smooth function $\chi$, whose precise form will be specified later, set
\begin{equation}
    \label{eq:def:xi}
    \xi^N(x)
      \coloneq
      \begin{cases}
          \chi \paren{\frac{\log{x}}{\log{N}}} &\text{ if }x>0,\\
          0 & \text{ if }x\le0,
      \end{cases}
\end{equation}
and define $u^N \coloneq \g^{(\alpha)} \xi^N$ for all $N\ge2$. We will use the notation $\xi^N_n \coloneq \xi^N(n)$ and abbreviate
\begin{equation}
    \label{eq:def:eta}
    \eta^N_{k,n} 
      \coloneq \frac{\grad^k (\g^{(\alpha)}\xi^N)_n}{\grad^k \g^{(\alpha)}_n}
      \quad\text{ for any } n\in\N \text{ and }k\in\{0,\dots,\ceil{\alpha}-1\}.
\end{equation}

Before proceeding to the derivation of the desired optimality properties, we first formulate two auxiliary lemmas that provide the necessary technical tools.

\begin{lemma}
    Using the notation introduced above, we have
    \begin{equation}
        \label{eq:lem:eta:expl}
        \eta_{k,n}^N
          =\sum_{j=0}^{k}\binom{k}{j} \frac{(n-j)_{j}}{(\alpha-k+1/2)_j}\,\grad^j\xi_n^N.
    \end{equation}
\end{lemma}
\begin{proof}
    First, we apply the discrete Leibniz rule
    \[
        \grad^k (uv)_n
          = \sum_{j=0}^k \binom{k}{j} \grad^{k-j} u_{n-j} \grad^j v_n,
          \quad\text{ where } k\in\N_0 \text{ and }n\in\Z,
    \]
    which can be easily proved by induction for all $u,v\in C(\Z)$, to $u=\g^{(\alpha)}$ and $v=\xi^N$. We obtain
    \[
        \eta^N_{k,n}
          = \sum_{j=0}^k \binom{k}{j} \frac{\grad^{k-j} \g^{(\alpha)}_{n-j}}{\grad^k \g^{(\alpha)}_n} \grad^j \xi^N_n.
    \]
    Recalling Lemma~\ref{lem:frac_grad_g} and definition~\eqref{eq:def:frac_g}, we find, for $n>j$, that
    \[
        \frac{\grad^{k-j} \g^{(\alpha)}_{n-j}}{\grad^k \g^{(\alpha)}_n}
          = \frac{\pochhammer{\alpha-k+j+1/2}{k-j}}{\pochhammer{\alpha-k+1/2}{k}} \frac{\g^{(\alpha-k+j)}_{n-j}}{\g^{(\alpha-k)}_n}
          = \frac{\pochhammer{n-j}{j}}{\pochhammer{\alpha-k+1/2}{j}}.
    \]
    If $n\le j$ the numerator on the left vanishes, while $\pochhammer{n-j}{j}=0$, so the same formula remains true. Combining the above identities gives~\eqref{eq:lem:eta:expl}.
\end{proof}

Next, let us extend the definition of the backward difference $\grad$ to functions by putting $\grad f(x) \coloneq f(x) - f(x-1)$. By induction in $k\in\N$, one readily verifies that the $k$-fold composition of $\grad$ admits the integral representation
\begin{equation}
    \label{eq:grad_int_repr}
    \grad^k f(x)
      = \int_{[0,1]^k} f^{(k)}(x-t_1-\dots-t_k) \dd t_1\dots\dd t_k,
\end{equation}
whenever $f$ is a $k$-times continuously differentiable function on $[x-k,x]$. Moreover, in view of~\eqref{eq:lem:eta:expl}, we denote
\[
    \eta^N_k(x)
      \coloneq \sum_{j=0}^{k}\binom{k}{j} \frac{(x-j)_{j}}{(\alpha-k+1/2)_j}\,\grad^j\xi^N(x)
      \quad\text{ for } x\in\R.
\]

\begin{lemma}
    \label{lem:eta'_bound}
    Using the notation introduced above, we have the bound
    \begin{equation}
        \label{eq:lem:eta'_bound}
        \abs{(\eta^N_k)'(x)}
          \le \frac{C_{\alpha,k}}{\log{N}} \frac{1}{x}
          \quad\text{ for all } x\ge k+1.
    \end{equation}
\end{lemma}
\begin{proof}
    For every $j\in\N$, repeated differentiation in~\eqref{eq:def:xi} yields the estimate
    \begin{equation}
        \label{eq:pr:lem:eta'_bound:xi_bound}
        \abs{(\xi^N)^{(j)}(x)}
          \le \frac{C_j}{\log{N}} \frac{1}{x^j}
          \quad\text{ for all }x>0.
    \end{equation}
    By employing the integral representation~\eqref{eq:grad_int_repr} and its derivative, we derive the upper bounds
    \[
        \abs{\grad^j \xi ^N(x)} \le \frac{C_j}{\log{N}} \frac{1}{x^j}
        \quad\text{ and }\quad
        \abs{\grad^j (\xi ^N)'(x)} \le \frac{C_j}{\log{N}} \frac{1}{x^{j+1}}
    \]
    for all $x\ge j+1$ and $j\in\N$ (the second estimate remains true even for $j=0$ by~\eqref{eq:pr:lem:eta'_bound:xi_bound}). Finally, since $\pochhammer{x-j}{j}$ is a polynomial in $x$ of degree $j$, we may bound this polynomial and its derivative by $C_j x^j$ and $C_j x^{j-1}$, respectively. Altogether, we obtain
    \begin{align*}
        \abs{(\eta^N_k)'(x)} 
          &\le \sum_{j=1}^k \binom{k}{j} \frac{C_j x^{j-1}}{\pochhammer{\alpha-k+1/2}{j}} \frac{C_j}{\log{N}} \frac{1}{x^j}
            + \sum_{j=0}^k \binom{k}{j} \frac{C_j x^{j}}{\pochhammer{\alpha-k+1/2}{j}} \frac{C_j}{\log{N}} \frac{1}{x^{j+1}} \\
          &\le \frac{C_{\alpha,k}}{\log{N}}\frac{1}{x}
    \end{align*}
    for all $x\ge k+1$. The proof of Lemma~\ref{lem:eta'_bound} is complete.
\end{proof}

We are now in a position to prove the optimality of the weight $\rho^{(\alpha)}$. By Theorem~\ref{thm:frac_birman_eq} and Lemma~\ref{lem:frac_grad_g}, we have the identity
\begin{equation}
    \label{eq:frac_hardy_eq_rho_a}
    \sum_{n=1}^\infty  \overline{u_n} (-\lap)^\alpha u_n - \sum_{n=1}^\infty  \rho^{(\alpha)}_n \abs{u_n}^2
          =\sum_{k=1}^{\ceil{\alpha}} R_k^{(\alpha)} (u)
\end{equation}
for all $u\in C_0(\N)$, where
\begin{numcases}
    {R_k^{(\alpha)} (u) \coloneq}
      \;\;\;\;\sum\limits_{n=1}^\infty  
        \scalebox{0.7}{$\rho^{(\alpha-k)}_{n+1}\grad^{k-1}\g^{(\alpha)}_{n+1} \grad^{k-1}\g^{(\alpha)}_{n}$} \abs{ \frac{\grad^{k-1}u_{n+1}}{\grad^{k-1}\g^{(\alpha)}_{n+1}} - \frac{\grad^{k-1}u_{n}}{\grad^{k-1}\g^{(\alpha)}_{n}}}^{2}
        & \text{ if }$k<\ceil{\alpha}$, \label{eq:rem_frac_hardy_eq_rho_a:1} \\
      -\!\!\!\sum\limits_{1\le n<m} \!\! \scalebox{0.8}{$(-\lap)^{\!\{\alpha\}}_{m,n}$} \,
        \scalebox{0.7}{$\grad^{\ceil{\alpha}-1}\g^{(\alpha)}_{m} \grad^{\ceil{\alpha}-1}\g^{(\alpha)}_{n}$} \abs{ \frac{\grad^{\ceil{\alpha}-1}u_{m}}{\grad^{\ceil{\alpha}-1}\g^{(\alpha)}_{m}} - \frac{\grad^{\ceil{\alpha}-1}u_{n}}{\grad^{\ceil{\alpha}-1}\g^{(\alpha)}_{n}}}^{2}
        & \text{ if }$k=\ceil{\alpha}$. \label{eq:rem_frac_hardy_eq_rho_a:2}
\end{numcases}

\hypertarget{pr:thm:optimality:a}{a)} \emph{Proof of criticality}:
Consider a fractional Hardy weight $\tilde{\rho}$ such that $\tilde{\rho}_n\ge\rho^{(\alpha)}_n$ for all $n\in\N$. Employing identity~\eqref{eq:frac_hardy_eq_rho_a} for the weight $\rho^{(\alpha)}$ and inequality~\eqref{eq:def:frac_hardy_weight} for $\tilde{\rho}$, we find that
\[
    0 
      \le \sum_{n=1}^\infty \paren{\tilde{\rho}_n-\rho^{(\alpha)}_n} \abs{u_n}^2
      \le \sum_{k=1}^{\ceil{\alpha}} R_k^{(\alpha)} (u)
\]
for all $u\in C_0(\N)$. We will use these inequalities for the regularisation of the parameter sequence $\g^{(\alpha)}$ in the form $u^N=\g^{(\alpha)}\xi^N$, where $\xi^N$ is defined by~\eqref{eq:def:xi} with the cut-off $\chi$ chosen as a smooth function satisfying 
\[
    \chi(x)
      =
      \begin{cases}
          1 & \text{ if } x\leq 1, \\
          0 & \text{ if } x\geq 2.
      \end{cases}
\]
Recall the notation~\eqref{eq:def:eta} and notice that, with this choice, we have
\begin{equation}
    \label{eq:pr:criticality:xi_eta_val}
    \xi^N_n =
      \begin{cases}
          1 & \text{ if } 1 \le n\le N, \\
          0 & \text{ if } \hspace{22pt} n\ge N^2,
      \end{cases}
    \quad\text{ and }\quad
    \eta^N_{k,n} =
      \begin{cases}
          1 & \text{ if } 1 \le n\le N, \\
          0 & \text{ if } \hspace{22pt} n\ge N^2+k
      \end{cases}
\end{equation}
for any $k\in\N_0$. Moreover, $\xi^N\to1$, and hence $u^N\to\g^{(\alpha)}$, pointwise, as $N\to\infty$. 

We will show that the remainders admit the estimate
\begin{equation}
    \label{eq:pr:criticality:rem_est}
    R^{(\alpha)}_k (u^N)
      \le \frac{C_\alpha}{\log{N}}
      \quad\text{ for all } k\in\{1,\dots,\ceil{\alpha}\} \text{ and } N>\ceil{\alpha}.
\end{equation}
Consequently, with the aid of Fatou's lemma, we find that
\[
    \sum_{n=1}^\infty \paren{\tilde{\rho}_{n}-\rho^{(\alpha)}_n} \abs{\g^{(\alpha)}_n}^2 
        = \sum_{n=1}^\infty \paren{\tilde{\rho}_{n}-\rho^{(\alpha)}_n} \lim_{N\to\infty}\abs{u_n^N}^2 
        \le \liminf_{N\to\infty} \sum_{n=1}^\infty  \paren{\tilde{\rho}_{n}-\rho^{(\alpha)}_n} \abs{u_n^N}^2 
        = 0.
\]
From the non-negativity of the summands on the left and positivity of $\g^{(\alpha)}_n$ for all $n\in\N$, ensured by assumption~\eqref{assump:A1}, we conclude that $\tilde{\rho}_n=\rho^{(\alpha)}_n$ for all $n\in\N$, which proves the criticality of $\rho^{(\alpha)}$.

It remains to establish the estimate~\eqref{eq:pr:criticality:rem_est}, which is to be done in the rest of the proof. We start by showing that
\begin{equation}
    \label{eq:pr:criticality:eta_est}
    \abs{\eta_{k,m}^N-\eta_{k,n}^N}
      \le \frac{C_{\alpha,k}}{\log{N}}\log\frac{m}{n}
      \quad\text{ for all } 1\le n<m
\end{equation}
and $k\in\{0,\dots,\ceil{\alpha}-1\}$. Assume first that $n\le N$. If, in addition, $n<m\le N$, then by~\eqref{eq:pr:criticality:xi_eta_val}, we have $\eta^N_{k,n}=\eta^N_{k,m}=1$, the left-hand side of~\eqref{eq:pr:criticality:eta_est} vanishes, and so the estimate holds trivially. If $n \le N <m$, we have $\eta^N_{k,n}=\eta^N_{k,N}=1$, and with the aid of Lemma~\ref{lem:eta'_bound}, we infer
\[
    \abs{\eta_{k,m}^N-\eta_{k,n}^N}
      = \abs{\eta_{k,m}^N-\eta_{k,N}^N}
      \le \int_N^m \abs{(\eta^N_k)'(x)}\dd x
      \le \frac{C_{\alpha,k}}{\log{N}} \log{\frac{m}{N}}
      \le \frac{C_{\alpha,k}}{\log{N}} \log{\frac{m}{n}}
\]
for all $N>\ceil{\alpha}\ge k+1$. Second, assume $n\ge N$, and take $m>n$. Similarly as above, the integration of \eqref{eq:lem:eta'_bound}, this time from $n$ to $m$, gives precisely~\eqref{eq:pr:criticality:eta_est}.

To complete the proof, we estimate the remainders $R_k^{(\alpha)}(u^N)$, defined by~\eqref{eq:rem_frac_hardy_eq_rho_a:1} and ~\eqref{eq:rem_frac_hardy_eq_rho_a:2}. Recall the asymptotics
\begin{equation}
    \label{eq:rho_g_asy}
    \rho^{(\alpha)}_n =\frac{\Gamma^2(\alpha+1/2)}{\pi}\frac{1}{n^{2\alpha}}\!\left[1+\mathcal{O}\!\left(\frac{1}{n}\right)\!\right]
    \quad\text{ and }\quad
    \g^{(\alpha)}_n = n^{\alpha-1/2}\left[1+\mathcal{O}\!\left(\frac{1}{n}\right)\!\right]
\end{equation}
for $n\to\infty$, which stem from~\eqref{eq:Gamma_ratio_exp} and formulas \eqref{eq:thm:opt_frac_hardy} and \eqref{eq:def:frac_g}. Consequently, employing~\eqref{eq:lem:frac:grad_g}, \eqref{eq:pr:criticality:xi_eta_val}, and~\eqref{eq:pr:criticality:eta_est} with $m=n+1$, or more precisely
\[
    \abs{\eta^N_{k,n+1}-\eta^N_{k,n}}
      \le \frac{C_{\alpha,k}}{\log{N}}\log\paren{1+\frac{1}{n}}
      \le \frac{C_{\alpha,k}}{\log{N}}\frac{1}{n},
\]
for $k\in\{1,\dots,\ceil{\alpha}-1\}$, we infer
\begin{align*}
    R_k^{(\alpha)}(u^N)
      & = \sum_{n=1}^{N^2+k} \rho^{(\alpha-k)}_{n+1} \grad^{k-1}\g^{(\alpha)}_{n+1} \grad^{k-1}\g^{(\alpha)}_{n} \abs{\eta^N_{k-1,n+1}-\eta^N_{k-1,n}}^2 \\
      & \le \frac{C_\alpha}{\log^2{N}} \sum_{n=1}^{N^2+k} \frac{1}{n^{2(\alpha-k)}} n^{2(\alpha-k+1/2)}\frac{1}{n^2} 
        \le \frac{C_\alpha}{\log^2{N}} \int_1^{N^2+k} \frac{\dd n}{n}
        \le \frac{C_\alpha}{\log{N}}
\end{align*}
for any $N>\ceil{\alpha}$. Now, suppose $k=\ceil{\alpha}$. If $\alpha\in\N$ is an integer, then $\{\alpha\}=1$, the double sum in~\eqref{eq:rem_frac_hardy_eq_rho_a:2} reduces as the only non-zero terms correspond to $m=n+1$, and one may proceed similarly as above to obtain
\[
    R_{\ceil{\alpha}}^{(\alpha)} (u^N) 
      = \sum_{n=1}^{N^2+\alpha} \grad^{\alpha-1}\g^{(\alpha)}_{n+1} \grad^{\alpha-1}\g^{(\alpha)}_{n} \abs{\eta^N_{\alpha-1,n+1}-\eta^N_{\alpha-1,n}}^2
      \le \frac{C_\alpha}{\log{N}}.
\]
Finally, if $\alpha\notin\N$, we make use of~\eqref{eq:pr:criticality:eta_est} and the bound~\eqref{eq:estimate_frac_lap} for the matrix elements $(-\lap)^{\{\alpha\}}_{m,n}$ to deduce that
\begin{align*}
    R_{\ceil{\alpha}}^{(\alpha)}(u^N)
      &\le C_\alpha  \sum_{n=1}^{N^2+\ceil{\alpha}} \sum_{j=1}^\infty \frac{1}{j^{2\{\alpha\}+1}}n^{\{\alpha\}-1/2}(n+j)^{\{\alpha\}-1/2}\,\frac{\log^2{(1+j/n)}}{\log^2{N}} \\
      &= \frac{C_\alpha}{\log^2{N}} \sum_{n=1}^{N^2+\ceil{\alpha}}\frac{1}{n^2} \sum_{j=1}^\infty \frac{(1+j/n)^{\{\alpha\}-1/2}}{(j/n)^{2\{\alpha\}+1}}\log^2{(1+j/n)}.
\end{align*}
Moreover, the inner sum can be further estimated by the corresponding integral, hence, after using the substitution $t=j/n$, we infer that
\[
    R_{\ceil{\alpha}}^{(\alpha)}(u^N)
      \le \frac{C_\alpha}{\log^2{N}} \sum_{n=1}^{N^2+\ceil{\alpha}}\frac{1}{n} \paren{\int_0^\infty \frac{(1+t)^{\{\alpha\}-1/2}}{t^{2\{\alpha\}+1}} \log^2{(1+t)}\dd t}
      \le \frac{C_\alpha}{\log{N}},
\]
since $\{\alpha\}\in(0,1)$ and the integral in parentheses converges to a positive constant. The proof of the estimate~\eqref{eq:pr:criticality:rem_est}, and therefore of the criticality of the weight $\rho^{(\alpha)}$ is complete.

b) \emph{Proof of optimality near infinity}:
Fix arbitrary $M\ge 1$. Recalling the alternative definition of optimality near infinity~\eqref{eq:rem:opt_near_inf}, we prove that the weight $\rho^{(\alpha)}$ enjoys this property by finding a sequence $\{ u^N\}_{N>M}\subset C_0(\N)$, such that $u^N_n=0$ for all $n<M$ and
\begin{equation}
    \label{eq:pr:opt_near_inf:def}
    \lim_{N\to\infty} \frac{\sum_{k=1}^{\ceil{\alpha}}R_k^{(\alpha)}(u^N)}{\sum_{n=1}^\infty\rho^{(\alpha)}_n\abs{u^N_n}^2}
      =0.
\end{equation}
This time, we set $u^N=\g^{(\alpha)}\xi^N$, with $\xi^N$ still defined by~\eqref{eq:def:xi}, but $\chi\in C_0^\infty(\R)$ is a bump function satisfying
\[
    \chi(x) =
    \begin{cases}
        0 & \text{ if } x \le 1 \text{ or } x \ge 4, \\
        1 & \text{ if } x \in [2,3].
    \end{cases}
\]
With this choice, we have
\[
    \xi^N_n =
      \begin{cases}
          0 & \text{ if } \hspace{31.5pt} n\le N, \\
          1 & \text{ if } N^2\le n\le N^3, \\
          0 & \text{ if } \hspace{31.5pt} n\ge N^4,
      \end{cases}
    \quad\text{ and }\quad
    \eta^N_{k,n} =
      \begin{cases}
          0 & \text{ if } \hspace{52.4pt} n\le N, \\
          1 & \text{ if } N^2+k \le n\le N^3, \\
          0 & \text{ if } \hspace{52.4pt} n\ge N^4+k.
      \end{cases}
\]
Therefore $u^N_n=0$ for all $n\le M$, whenever $N>M$.

Proceeding almost verbatim as in the proof of criticality~\hyperlink{pr:thm:optimality:a}{a)}, it is straightforward to verify that the estimates~\eqref{eq:pr:criticality:eta_est} still hold, hence we may bound the remainders as
\[
    R_k^{(\alpha)}(u^N)
      \le \frac{C_\alpha}{\log^2{N}} \sum_{n=1}^{N^4+k} \frac{1}{n} \le \frac{C_\alpha}{\log{N}}
\]
for all $k\in\{1,\dots,\ceil{\alpha}\}$ and $N>\max\{M,\ceil{\alpha}\}$. On the other hand, making use of the formulas in~\eqref{eq:rho_g_asy}, the denominator in~\eqref{eq:pr:opt_near_inf:def} can be estimated from below by
\[
    \sum_{n=1}^\infty \rho^{(\alpha)}_n \abs{u^N_n}^2 
      \ge \sum_{n=N^2}^{N^3} \rho^{(\alpha)}_n \abs{\g^{(\alpha)}_n}^2
      \ge C_\alpha \sum_{n=N^2}^{N^3} \frac{1}{n^{2\alpha}} n^{2\alpha-1}
      \ge C_\alpha \log{N}.
\]
Combining these two estimates implies~\eqref{eq:pr:opt_near_inf:def} and completes the proof of optimality near infinity.

c) \emph{Proof of non-attainability}:
In order to prove the non-attainability, suppose that $u\in \ell^2(\N)$ attains equality in~\eqref{eq:thm:abs_frac_birman_ineq} with $\rho(\g)=\rho^{(\alpha)}$. Consider a sequence $\{u^N\}_{N\in\N} \subset C_0(\N)$ of compactly supported sequences approximating $u$, i.e. $u^N\to u$ as $N\to\infty$ in the norm of $\ell^2(\N)$, and thus also pointwise. Identity~\eqref{eq:frac_hardy_eq_rho_a} implies
\[
    0\le \sum_{k=1}^{\ceil{\alpha}} R_k^{(\alpha)}(u^N)
      = \sum_{n=1}^\infty \overline{u^N_n}(-\lap)^\alpha u^N_n - \sum_{n=1}^\infty \rho^{(\alpha)}_n\abs{u^N_n}^2
\]
for all $N\ge1$. Since both $(-\lap)^\alpha$ and $\rho^{(\alpha)}$ determine bounded operators on $\ell^2(\N)$ (see Remark~\ref{rem:def:frac_hardy_weight}), the right-hand side of the identity tends to zero as $N\to\infty$. Consequently, the non-negative remainder $R_k^{(\alpha)}(u^N)\to0$, as $N\to\infty$, for any $k\in\{1,\dots,\ceil{\alpha}\}$. Consider the particular case $k=\ceil{\alpha}$. Since the prefactor
\[
-(-\lap)^{\!\{\alpha\}}_{m,n}\,\grad^{\ceil{\alpha}-1}\g^{(\alpha)}_{m} \grad^{\ceil{\alpha}-1}\g^{(\alpha)}_{n}
\]
in \eqref{eq:rem_frac_hardy_eq_rho_a:2}
is non-negative for all $1\leq n<m$ and positive for $m=n+1$, we have the estimate
\[
R_{\ceil{\alpha}}^{(\alpha)}(u^N)\geq-(-\lap)^{\!\{\alpha\}}_{n+1,n}\,\grad^{\ceil{\alpha}-1}\g^{(\alpha)}_{n+1} \grad^{\ceil{\alpha}-1}\g^{(\alpha)}_{n}
\left|\frac{\grad^{\ceil{\alpha}-1}u_{n+1}^{N}}{\grad^{\ceil{\alpha}-1}\g^{(\alpha)}_{n+1}} -\frac{\grad^{\ceil{\alpha}-1}u_{n}^{N}}{\grad^{\ceil{\alpha}-1}\g^{(\alpha)}_{n}}\right|^{2},
\]
from which, after sending $N\to\infty$, we deduce the necessary conditions
\[
 \frac{\grad^{\ceil{\alpha}-1}u_{n+1}}{\grad^{\ceil{\alpha}-1}\g^{(\alpha)}_{n+1}} -\frac{\grad^{\ceil{\alpha}-1}u_{n}}{\grad^{\ceil{\alpha}-1}\g^{(\alpha)}_{n}}=0
    \quad\text{ for all }n\in\N,
\]
and $u_n=0$ for all $n\le0$, by our zero-extension convention. The solution space of this difference equation is one-dimensional and spanned by $\g^{(\alpha)}$, hence $u=c\,\g^{(\alpha)}$ for a constant $c\in\C$. By~\eqref{eq:rho_g_asy}, such $u$ can belong to $\ell^{2}(\N)$ only if $c=0$, i.e. $u\equiv0$. The proof of non-attainability, and consequently of optimality and of Theorem~\ref{thm:opt_frac_hardy} is complete.
\qed

\subsection{Proof of Theorem~\ref{thm:clas_frac_hardy}}
\label{subsec:pr:clas_ineq}

We already know from Theorem~\ref{thm:opt_frac_hardy} that
\[
    \sum_{n=1}^\infty \overline{u_n} (-\lap)^\alpha u_n
      > \sum_{n=1}^\infty \rho^{(\alpha)}_n\abs{u_n}^2
\]
for all non-trivial $u\in\ell^{2}(\N)$. The strict inequality is a consequence of the non-attainability of $\rho^{(\alpha)}$. In this proof, we will show that
\begin{equation}
    \label{eq:lower_bound_for_rho_a}
    \rho^{(\alpha)}_n 
      \geq \frac{\Gamma^2(\alpha+1/2)}{\pi} \frac{1}{(n+\ceil{\alpha}-1)^{2\alpha}}
      \quad\text{ for all } n\ge1.
\end{equation}
Combining the last two inequalities, we get
\[
    \sum_{n=1}^\infty \overline{u_n} (-\lap)^\alpha u_n
      > \frac{\Gamma^2(\alpha+1/2)}{\pi} \sum_{n=1}^\infty \frac{\abs{u_n}^2}{(n+\ceil{\alpha}-1)^{2\alpha}}
\]
for any non-trivial $u\in\ell^{2}(\N)$. Shifting the index $n$ then immediately yields the inequality \eqref{eq:thm:clas_frac_hardy} for any $u\in\ell^{2}(\N)$ satisfying $u_n=0$ for all $n<\ceil{\alpha}$. Moreover, the inequality is strict unless $u\equiv0$. The optimality of the constant $\Gamma^2(\alpha+1/2)/\pi$ is a consequence of the asymptotics of weight $\rho^{(\alpha)}$, see~\eqref{eq:rho_g_asy}, and point \eqref{it:rem:optimality:3} of Remark~\ref{rem:optimality}.

It remains to verify the inequality~\eqref{eq:lower_bound_for_rho_a}. Bearing formula~\eqref{eq:thm:opt_frac_hardy} in mind, it is equivalent to showing that
\begin{equation}
    \frac{\Gamma(2n+1)}{\Gamma(2n+2\alpha+1)} \geq \frac{1}{(2n+2\ceil{\alpha})^{2\alpha}}
    \quad\text{ for all }n\in\N_0.
\label{eq:gamma_ratio_ineq}
\end{equation}
With the aid of \textit{Gautschi's inequality}~\cite[Eq.~(5.6.4)]{DLMF}, which implies
\[
\frac{\Gamma(x)}{\Gamma(x+s)}\geq\frac{1}{x^{s}} \quad\mbox{ for all } x>0 \mbox{ and } s\in(0,1],
\]
we deduce, for any $n\in\N_0$ and $\alpha>0$, that
\begin{align*}
 \frac{\Gamma(2n+1)}{\Gamma(2n+2\alpha+1)}&=\frac{1}{(2n+1)(2n+2)\cdots(2n+\ceil{2\alpha}-1)}\frac{\Gamma(2n+\ceil{2\alpha})}{\Gamma(2n+\ceil{2\alpha}+\{2\alpha\})} \\
 &\geq\frac{1}{(2n+\ceil{2\alpha})^{\ceil{2\alpha}-1}}\frac{1}{(2n+\ceil{2\alpha})^{\{2\alpha\}}}
 =\frac{1}{(2n+\ceil{2\alpha})^{2\alpha}},
\end{align*}
from which the inequality \eqref{eq:gamma_ratio_ineq} follows since $\ceil{2\alpha}\leq2\ceil{\alpha}$. The proof of Theorem~\ref{thm:clas_frac_hardy} is complete.
\qed

\begin{remark}
    The inequality in~\eqref{eq:gamma_ratio_ineq}, and therefore also in~\eqref{eq:lower_bound_for_rho_a}, can be shown to be strict. In fact, one has a slightly stronger estimate
    \[
     \frac{\Gamma(2n+1)}{\Gamma(2n+2\alpha+1)} > \frac{1}{(2n+\alpha+1)^{2\alpha}}
    \quad\text{ for all }n\in\N_0 \mbox{ and }\alpha>0.
    \]
\end{remark}

\subsection*{Acknowledgement}
This research was supported by the CTU Future Fund (Project ID: CVUT-BrF-26-23028S).

\bibliographystyle{acm}

\end{document}